\documentclass[11pt,a4paper]{article}
\usepackage{amsmath,amsfonts,amssymb,latexsym,graphics,epsfig}
\usepackage{color}
\usepackage{amsthm}
\usepackage{graphicx,url}
\usepackage{hyperref}
\usepackage[english]{babel}
\usepackage{epsfig}
\usepackage{enumerate}
\usepackage{graphicx}
\usepackage{floatrow}
\usepackage{float}
\usepackage{pifont}
\usepackage{caption}
\usepackage{subfig}
\usepackage{fullpage}
\usepackage{fancyhdr}
\usepackage{authblk}
\usepackage[arrow,frame,matrix]{xy}

\newtheorem{thm}{Theorem}[section]
\newtheorem{lem}[thm]{Lemma}
\newtheorem{pro}[thm]{Proposition}
\newtheorem{de}[thm]{Definition}
\newtheorem{re}[thm]{Remark}

\newtheorem{cor}[thm]{Corollary}

\def\be{\begin{equation}}
\def\ee{\end{equation}}
\def\bea{\begin{eqnarray}}
\def\eea{\end{eqnarray}}
\numberwithin{equation}{section}

\begin{document}
\title{On the integrability of strongly regular graphs}

\author{Jack H. Koolen\thanks{ J.H. Koolen is partially supported by the National Natural Science Foundation of China (Grant No. 11471009 and Grant No. 11671376).}~, Masood Ur Rehman\thanks{ M.U. Rehman is supported by the Chinese Scholarship Council at USTC, China.}~, Qianqian Yang\thanks{ Corresponding author}}

%

\maketitle
\date{}
\begin{abstract}
Koolen et al. showed that if a connected graph with smallest eigenvalue at least $-3$ has large minimal valency, then it is $2$-integrable. In this paper, we will prove that a lower bound for the minimal valency is $166$.
\end{abstract}



\section{Introduction}
(For undefined notations, we refer to next section.) In 1976, Cameron, Goethals, Seidel and Shult \cite{PJC} showed that:
\begin{thm}
If $G$ is a connected graph with smallest eigenvalue at least $-2$, then either $G$ has at most $36$ vertices, or $G$ is integrable.
\end{thm}

Note that for a connected graph $G$, the graph $G$ is complete if and only if $\theta_{\min}(G)\geq-1$; for $-1>\theta_{\min}(G)\geq-2$, Cameron et al. showed that the lattice $\Lambda(G)$ is one of the root lattices $A_p, D_q,E_6,E_7,E_8$ with $p\geq1$ and $q\geq4$, and note that $A_p$ and $D_q$ are integrable, and $E_6,E_7,E_8$ are $2$-integrable, but not integrable, see \cite{CS2}.

Thus as a corollary of (the proof of) the result of Cameron et al. \cite{PJC}, we have
\begin{cor}
Any connected graph $G$ with smallest eigenvalue at least $-2$ is $2$-integrable. Moreover, if $G$ is not integrable, then $G$ has at most $36$ vertices.
\end{cor}

For the case of graphs with smallest eigenvalue at least $-3$, Koolen, Yang and Yang \cite{kyy3} generalized a result of Hoffman (1977) \cite{Hof77} and showed the following.

\begin{thm}\label{kyy3}
There exists a positive integer $\kappa$ such that if a connected graph $G$ has smallest eigenvalue at least $-3$ and minimal valency at least $\kappa$, then $G$ is $2$-integrable.
\end{thm}

In this paper, we will show that the number $\kappa$ in Theorem \ref{kyy3} is at least $166$. Our main result is as follows, which immediately follows from Theorems \ref{mainintegrable}, \ref{complementofsg}, \ref{HS}, \ref{gq39}, \ref{mc} and Proposition \ref{mcinformation} {\rm (ii)}.

\begin{thm}\label{mainthm}
\begin{enumerate}
\item If $G$ is an integrable strongly regular graph with parameters $(v,k,\lambda,\mu)$ and smallest eigenvalue $\theta_{\min}(G)\in[-3,-2)$, then either $\mu\leq9$ or $G$ is the strongly regular graph $K_{\frac{v}{3}\times 3}$.
\item The strongly regular graph $\overline{SG}$, which is the complement of the Sims-Gewirtz graph, with parameters $(56,45,36,36)$ and smallest eigenvalue $-3$ is not integrable, but
$2$-integrable.
\item The Hoffman-Singleton graph $HoSi$ with parameters $(50,7,0,1)$ and smallest eigenvalue $-3$ is not $2$-integrable.
\item The strongly regular graph $\overline{GQ}(3,9)$, which is the complement of the point graph of the generalized quadrangle of order $(3,9)$, with parameters $(112,81,60,54)$ and smallest eigenvalue $-3$ is not $2$-integrable, but $4$-integrable.
\item The strongly regular graph $\overline{McL}$, which is the complement of the McLaughlin graph, with parameters $(275,162,105,81)$ and smallest eigenvalue $-3$ is not $2$-integrable, but $4$-integrable.
\end{enumerate}
\end{thm}

\begin{re}
\begin{enumerate}
\item The result $\mu\leq9$ in Theorem \ref{mainthm} {\rm (i)} can not be improved, as the strongly regular graph with parameters $(35,18,9,9)$ and smallest eigenvalue $-3$ is integrable.
\item Adding three distinct vertices $x_1,x_2,x_3$ to the graph $\overline{McL}$ such that $x_1,x_2,x_3$ are adjacent to each other and to all vertices of $\overline{McL}$, we obtain a new graph with smallest eigenvalue $-3$ and minimal valency $166$. This shows that the value of $\kappa$ in Theorem \ref{kyy3} is at least $166$.
\end{enumerate}
\end{re}

\section{Definitions and preliminaries}
All the graphs considered in this paper are finite, undirected and simple. Let $G=(V(G), E(G))$ be a graph with the vertex set $V(G)$ and the edge set $E(G)$. Recall that the \emph{adjacency matrix} $A(G)$ of $G$ is the $01$-matrix with rows and columns indexed by the vertices of $G$, such that the $xy$-entry of $A(G)$ is $1$ if and only if $x$ and $y$ are adjacent. We write $x\sim y$ if the vertices $x$ and $y$ are adjacent and $x\not\sim y$ otherwise. The \emph{eigenvalues} of $G$ are the eigenvalues of $A(G)$ and we denote by $\theta_{\min}(G)$ the smallest eigenvalue of $G$. For a vertex $x$ in G, we write $k(x)$ for the valency of $x$, that is, the number of neighbors of $x$. We call the graph $G$ \emph{$k$-regular} if $k(x)=k$ for any vertex $x$.

Recall that a \emph{$c$-clique} (resp. \emph{$c$-coclique}) is a subgraph of $G$ with $c$ vertices such that every pair of vertices are adjacent (resp. not adjacent). The well-known Ramsey's Theorem is as follows.

\begin{thm}\label{Ramsey}
(\cite{Ramsey}) Let a, b be two positive integers. Then there exists a minimum positive integer $R(a,b)$ such that for any graph $G$ on $m\geq R(a,b)$ vertices, the graph $G$ contains a $a$-clique or a $b$-coclique as a subgraph.
\end{thm}

Let $\pi=\{V_1,\ldots,V_p\}$ be a partition of the vertex set $V(G)$ of the graph $G$. For each vertex $x$ in $V_i$, write $d_x^{(j)}$ for the number of neighbors of $x$ in $V_j$. Then we write $q_{i,j}=\frac{1}{|V_i|}\sum_{x\in V_i}d_x^{(j)}$ for the average number of neighbors in $V_j$ of vertices in $V_i$. The matrix $Q_{\pi}:=(q_{i,j})$ is called the \emph{quotient matrix} of $G$ relative to $\pi$ and $\pi$ is called \emph{equitable} if for all $i$ and $j$, we have $d_x^{(j)}=q_{i,j}$ for each $x\in V_i$.

\begin{lem}(\cite[Theorem $9.3.3$ and Lemma $9.6.1$]{Ch&G})\label{quotienteigenvalues1}
If $\pi$ is a partition of the vertex set of a graph $G$ which has three distinct eigenvalues, then $\pi$ is equitable if and only if all eigenvalues of the quotient matrix $Q_\pi$ are the eigenvalues of $G$.
\end{lem}

Now let us look at a class of highly structured graphs, that is, the class of strongly regular graphs. A graph $G$ with $v$ vertices is said to be a \emph{strongly regular graph} with parameters $(v,k,\lambda,\mu)$ if it is $k$-regular, every pair of adjacent vertices has $\lambda$ common neighbors, and every pair of distinct nonadjacent vertices has $\mu$ common neighbors.

Note that if $G$ is a strongly regular graph with parameters $(v,k,\lambda,\mu)$, then the complement $\overline{G}$ of the graph $G$ is also strongly regular and its parameters are
\begin{equation}\label{compleparameter}
(v,v-1-k,v-2-2k+\mu,v-2k+\lambda).
\end{equation}

Let $G$ be a strongly regular graph with parameters $(v, k, \lambda, \mu)$. For any vertex $x$ of $G$, we denote by $N_i(x)$ the set of vertices at
distance $i$ from $x$, where $i=1,2$. We call the graph induced by $N_1(x)$ (resp. $N_2(x)$)  the \emph{first} (resp. \emph{second}) \emph{subconstituent} of $G$ relative to $x$. Let $\pi(x):=\{\{x\},N_1(x),N_2(x)\}$ be a partition of $V(G)$. Note that $\pi(x)$ is equitable and the quotient matrix $Q_{\pi(x)}$ of $G$ relative to $\pi$ is as follows:
\begin{equation}\label{quotientmatrix}
Q_{\pi(x)}=\begin{pmatrix}
0 & k & 0 \\
1 & \lambda & k-1-\lambda\\
0 & \mu & k-\mu
\end{pmatrix}.
\end{equation}
By Lemma \ref{quotienteigenvalues1}, the eigenvalues of the quotient matrix $Q_{\pi(x)}$ are also the eigenvalues of $G$.

\subsection{Graphs and lattices}
In this subsection, we will define the $s$-integrability of graphs. Before that, we will introduce integral lattices and their $s$-integrability.

Let $\Lambda$ be a subset of $\mathbb{R}^n$. We say that $\Lambda$ is a \emph{lattice} if there exist vectors $\mathbf{u}_1,\mathbf{u}_2,\ldots,\mathbf{u}_m\in\mathbb{R}^n$ (for some $m$) such that $\Lambda=\big\{\sum a_i\mathbf{u}_i\mid a_i\in\mathbb{Z}\big\}$ and we call $\{\mathbf{u}_1,\ldots,\mathbf{u}_m\}$ a \emph{generator set} of the lattice $\Lambda$.

A lattice $\Lambda$ is called \emph{integral}, if $(\mathbf{v}_i,\mathbf{v}_j)\in\mathbb{Z}$ for all $\mathbf{v}_i,\mathbf{v}_j\in\Lambda$. An integral lattice $\Lambda$ is called \emph{$s$-integrable} (for some positive integer $s$) if $\sqrt{s}\Lambda$ can be embedded in the standard lattice. Note that a $1$-integrable lattice is also called \emph{integrable}.

For a positive real number $t$, a \emph{representation of norm $t$} of a graph $G$ is a map
$\varphi : V(G) \mapsto \mathbb{R}^n$ (for some positive integer $n$) such that

\[
(\varphi(x),\varphi(y) )=
\left\{
  \begin{array}{ll}
    $t$ &\text{ if $x=y$}, \\
    1 &\text{ if $x \sim y$}, \\
    0 &\text{ otherwise}. \\
  \end{array}
\right.
\]
Note that $G$ has a representation of norm $t$ if and only if $\theta_{\min}(G)\geq -t$, as $A(G)+tI$ has to be positive semidefinite.

Let $\varphi$ be a representation of norm $-\lfloor\theta_{\min}(G)\rfloor$ of $G$. The integral lattice generated by $\varphi$ is the lattice $\Lambda^{\varphi}(G):=\big\{\sum a_x\varphi(x)\mid a_x\in\mathbb{Z}, x\in V(G)\big\}$. As the Gram matrix of its generator set $\{\varphi(x)\mid x\in V(G)\}$ is equal to $A(G)-\lfloor\theta_{\min}(G)\rfloor I$, (the isomorphic class of) the lattice $\Lambda^{\varphi}(G)$ only depends on the matrix $A(G)-\lfloor\theta_{\min}(G)\rfloor I$, not on the particular representation $\varphi$. Thus, we denote this lattice by $\Lambda(G)$ and we say the graph $G$ is \emph{$s$-integrable} if the lattice $\Lambda(G)$ is $s$-integrable. If the lattice $\Lambda(G)$ is $1$-integrable, we also say that the graph $G$ is \emph{integrable}.

A graph $G$ is $s$-integrable if and only if there exists an integral matrix $N$ such that the equality $A(G)-\lfloor\theta_{\min}(G)\rfloor I=\frac{1}{s}N^TN$ holds.

Note that if $G$ is $s$-integrable, it is also $\tau s$-integrable for all positive integers $\tau$. The smallest positive integer $s$ such that a graph is $s$-integrable gives a measure of the complexity of the graphs structure.

\subsection{Designs}
\begin{de}
A \emph{design} is a pair $(S,\mathcal{B})$, where $S$ is a set of elements called \emph{points} and $\mathcal{B}$ is a collection of nonempty subsets of $S$ called \emph{blocks}. Let $\tilde{v},\tilde{k},\tilde{\lambda}$ and $\tilde{t}$ be positive integers such that $\tilde{v}>\tilde{k}\geq \tilde{t}$. A \emph{$\tilde{t}$-$(\tilde{v},\tilde{k},\tilde{\lambda})$} design is a design $(S,\mathcal{B})$ such that the following properties are satisfied:
\begin{enumerate}
\item $|S|=\tilde{v}$,
\item each block contains exactly $\tilde{k}$ points, and
\item every set of $\tilde{t}$ distinct points is contained in exactly $\tilde{\lambda}$ blocks.
\end{enumerate}
\end{de}
\begin{lem}(\cite[Theorem 9.4]{Stin})\label{designlem}
Suppose $(S,\mathcal{B})$ is a $\tilde{t}$-$(\tilde{v},\tilde{k},\tilde{\lambda})$ design and $1\leq \tilde{s}\leq\tilde{t}$. Then $(S,\mathcal{B})$ is an $\tilde{s}$-$(\tilde{v},\tilde{k},\tilde{\lambda}_{\tilde{s}})$ design, where
\begin{equation}\label{lambdas}
\tilde{\lambda}_{\tilde{s}}=\frac{\tilde{\lambda}\binom{\tilde{v}-\tilde{s}}{\tilde{t}-\tilde{s}}}{\binom{\tilde{k}-\tilde{s}}{\tilde{t}-\tilde{s}}}.
\end{equation}
\end{lem}
If $\tilde{\lambda}=1$, then the $\tilde{t}$-$(\tilde{v},\tilde{k},1)$ design is called a \emph{Steiner system} and denote by $S(\tilde{t},\tilde{k},\tilde{v})$. Note that a Steiner system
$S(\tilde{t},\tilde{k},\tilde{v})$ with $\tilde{t}\geq2$ is also a $2$-$(\tilde{v},\tilde{k},\tilde{\lambda}_{2})$ design, where $\tilde{\lambda}_2=\frac{\binom{\tilde{v}-\tilde{s}}{\tilde{t}-\tilde{s}}}{\binom{\tilde{k}-\tilde{s}}{\tilde{t}-\tilde{s}}}.$

\begin{de}
A $2$-$(\tilde{v},\tilde{k},\tilde{\lambda})$ design is called \emph{quasi-symmetric} if the cardinality of the intersection of two blocks takes only two values.
\end{de}

\begin{pro}(\cite[Theorem 3.7]{s24intersection})\label{s23}
The Steiner system $S(4,7,23)$ is quasi-symmetric with intersection numbers $1$ and $3$.
\end{pro}

Let $(S,\mathcal{B})$ be a quasi-symmetric design with intersection numbers $l_{1}$ and $l_{2}$, where $l_{1}<l_{2}$. The \emph{block graph} of $(S,\mathcal{B})$ is the graph whose vertices are the blocks of $(S,\mathcal{B})$, with two vertices adjacent if and only if their intersection has cardinality $l_2$.
\begin{lem}(cf. \cite[Theorem 3.2]{camerondesign})\label{quasisymmetric}
The block graph of a quasi-symmetric design is strongly regular.
\end{lem}

Let $(S,\mathcal{B})$ be a design where $S=\{x_1,\ldots,x_{\tilde{v}}\}$ and $\mathcal{B}=\{B_1,\ldots,B_{\tilde{b}}\}$. The \emph{incidence matrix} of $(S,\mathcal{B})$ is the $\tilde{v}\times \tilde{b}$ $01$-matrix $\mathcal{I}=(\mathcal{I}_{x_i,B_j})$ defined by the rule $\mathcal{I}_{x_i,B_j}$ equals $1$ if $x_i\in B_j$ and $0$ otherwise.

\section{Basic results on $s$-integrable graphs}
In this section, we will give several results about the $s$-integrability of graphs. Some of the results play an important role in the proof of our main theorem. Before stating our results, the following notations are necessary. For a given vector $\mathbf{r}$, we denote by ${\rm supp}(\mathbf{r})$ the support of $\mathbf{r}$, that is, ${\rm supp}(\mathbf{r})=\{i \mid \mathbf{r}_i\neq0\}$. If a vector $\mathbf{r}$ is indexed by the vertex set of a graph $G$, we denote by $G(\mathbf{r})$ the subgraph of $G$ induced by the set ${\rm supp}(\mathbf{r})$. For convenience, we also write
$$\gamma_x^{\mathbf{r}}:=\mathbf{r}_x,~\delta_x^{\mathbf{r}}:=\sum_{y\sim x}\mathbf{r}_y,~\zeta_x^{\mathbf{r}}:=\sum_{y\neq x,y\not\sim x}\mathbf{r}_y,$$ and $$\sigma^{\mathbf{r}}:=\gamma_x^{\mathbf{r}}+\delta_x^{\mathbf{r}}+\delta_x^{\mathbf{r}}.$$

Now we start our work.
\begin{pro}\label{basicpro}
Assume that $G$ is an $s$-integrable graph with smallest eigenvalue $\theta_{\min}(G)$ and $N$ is a $\rho\times|V(G)|$ integral matrix satisfying $A(G)+\lceil-\theta_{\min}(G)\rceil I=\frac{1}{s}N^TN$. Then:
\begin{enumerate}
\item There exists a row vector $\mathbf{r}$ of $N$ such that
\begin{equation}\label{supportsize}
|{\rm supp}(\mathbf{r})|\leq\frac{s\lceil-\theta_{\min}(G)\rceil |V(G)|}{\rho}\leq\frac{s\lceil-\theta_{\min}(G)\rceil |V(G)|}{{\rm rank}(N)}
\end{equation}holds, where ${\rm rank}(N)$ is the rank of $N$.
\item Let $\pi:=\{V_1,\ldots,V_p\}$ be an equitable partition of the vertex set of $G$ such that the quotient matrix $Q_\pi$ has $\theta_{\min}(G)$ as its smallest eigenvalue. If $\theta_{\min}(G)$ is an integer, then for any row $\mathbf{r}$ of $N$ and any eigenvector $\mathbf{u}$ of $Q_\pi$ with eigenvalue $\theta_{\min}(G)$, we have
    \begin{equation}\label{stronglysupport}
    \sum_{i=1}^p\mathbf{u}_i\sum_{x\in V_i}\mathbf{r}_x=0, \text{ where }\mathbf{u}_i   \text{ is the }V_i\text{-entry of } \mathbf{u}.
   \end{equation}
\item For any row $\mathbf{r}$ of $N$, we have the following inequality:
     \begin{equation}\label{normofr}
    \sum_{x\in V(G)}\gamma_x^{\mathbf{r}}\delta_x^{\mathbf{r}}\geq\frac{1}{s}(\mathbf{r}\mathbf{r}^T)^2+\lfloor\theta_{\min}(G)\rfloor\mathbf{r}\mathbf{r}^T.
   \end{equation}
\item If $G'$ is an induced subgraph of $G$ with $\lfloor\theta_{\min}(G')\rfloor=\lfloor\theta_{\min}(G)\rfloor$, then $G'$ is also $s$-integrable.
\end{enumerate}
\end{pro}
\begin{proof}
\begin{enumerate}
\item Suppose $N$ has $b$ non-zero entries. For any column $N_x$ of $N$, it has at most $s\lceil-\theta_{\min}(G)\rceil$ non-zero entries since $N$ is integral and $(N_x,N_x)=s\lceil-\theta_{\min}(G)\rceil$. Thus, $b\leq s\lceil-\theta_{\min}(G)\rceil|V(G)|$. Choose $\mathbf{r}$ to be a row of $N$ whose support size is the smallest. It is easy to see that $b\geq|\rm{supp}(\mathbf{r})|\rho\geq|\rm{supp}(\mathbf{r})|{\rm rank}(N)$ and hence ({\rm i}) holds.
\item Define a vector $\widetilde{\mathbf{u}}$ such that $\widetilde{\mathbf{u}}_x=\mathbf{u}_i$ if $x\in V_i$. It is easy to check that $A(G)\widetilde{\mathbf{u}}=\theta_{\min}(G)\widetilde{\mathbf{u}}$ holds and thus $N\widetilde{\mathbf{u}}=\mathbf{0}$. Clearly, for the row $\mathbf{r}$, $\mathbf{r}\widetilde{\mathbf{u}}=\sum_{x\in V(G)}\mathbf{r}_x\widetilde{\mathbf{u}}_x=\sum_{i=1}^p\mathbf{u}_i\sum_{x\in V_i}\mathbf{r}_x=0$ and \rm(ii) holds.
\item Since $\frac{1}{s}(\mathbf{r}\mathbf{r}^T)^2\leq\frac{1}{s}\mathbf{r}N^TN\mathbf{r}^T=\mathbf{r}(A(G)+\lceil-\theta_{\min}(G)\rceil I)\mathbf{r}^T=\mathbf{r}A(G)\mathbf{r}^T-\lfloor\theta_{\min}(G)\rfloor \mathbf{r}\mathbf{r}^T$, we have $\frac{1}{s}(\mathbf{r}\mathbf{r}^T)^2+\lfloor\theta_{\min}(G)\rfloor\mathbf{r}\mathbf{r}^T\leq\mathbf{r}A(G)\mathbf{r}^T=\sum_{x\in V(G)}\mathbf{r}_x\sum_{y\sim x}\mathbf{r}_y=\sum_{x\in V(G)}\gamma_x^{\mathbf{r}}\delta_x^{\mathbf{r}}$ and \rm{(iii)} holds.
\item This is clear.
\end{enumerate}
\end{proof}

\begin{pro}\label{coclipro}
Let $G$ be a graph with smallest eigenvalue $\theta_{\min}(G)\in[-3,-2)$. If $G$ is $2$-integrable, then for any integral matrix $N$ satisfying $A(G)+3I=\frac{1}{2}N^TN$ and any row $\mathbf{r}$ of $N$, the order of a coclique in the subgraph $G(\mathbf{r})$ is at most $6$.

Moreover, if $|\mathbf{r}_x|=2$ for some $x\in V(G)$, then the order of a coclique in the subgraph $G(\mathbf{r})$ is at most $3$.
\end{pro}
\begin{proof}Assume $\mathbf{r}$ is the $r^{\rm{th}}$ row of $N$ and $\overline{C}$ is a coclique in $G(\mathbf{r})$ with $V(\overline{C})=\{x_1,\ldots,x_{\bar{c}}\}$, where $\bar{c}$ is the order of $\overline{C}$. Denote by $\widetilde{N}$ the matrix whose columns are indexed by the vertices of $\overline{C}$ such that its $x_i$-column equals $\frac{\mathbf{r}_{x_i}}{|\mathbf{r}_{x_i}|}N_{x_i}$, where $N_{x_i}$ is the $x_i$-column of $N$. By deleting the $r^{{\rm th}}$ row from $\widetilde{N}$, we obtain its submatrix $\widetilde{N}^\prime$. Note that $N^TN=2A+6I$ and there exists at most one vertex $x$ in $\overline{C}$ such that $|\mathbf{r}_{x}|=2$. If $|\mathbf{r}_x|=2$ for some $x$, then $(\widetilde{N}^\prime)^T\widetilde{N}^\prime=\left(\begin{array}{cc}2 & -2\mathbf{j}_{\bar{c}-1}^T \\-2\mathbf{j}_{\bar{c}-1} & 6I_{\bar{c}-1}-J_{\bar{c}-1}\end{array}\right)$;
otherwise $(\widetilde{N}^\prime)^T\widetilde{N}^\prime=6I_{\bar{c}}-J_{\bar{c}}$, where $\mathbf{j}$ and $J$ are the all-ones vector and all-ones matrix respectively. Since the matrix $(\widetilde{N}^\prime)^T\widetilde{N}^\prime$ is positive semidefinite, for the former case we have $\bar{c}\leq3$, since the matrix $\left(\begin{array}{cc}2 & -2\mathbf{j}_{3}^T \\-2\mathbf{j}_{3} & 6I_{3}-J_{3}\end{array}\right)$ is not positive semidefinite; for the latter case we have $\bar{c}\leq6$, since the matrix $6I_{7}-J_{7}$ is not positive semidefinite. This completes the proof.
\end{proof}

\section{Integrable strongly regular graphs}

\subsection{Two classes of integrable strongly regular graphs}
In this subsection, we will introduce geometric strongly regular graphs, and prove that they are integrable. Moreover, we will show that the regular complete multipartite graphs are also integrable at the end of this subsection.

Let $G$ be a strongly regular graph with parameters $(v,k,\lambda,\mu)$ and smallest eigenvalue $\theta_{\min}(G)$. In 1973, Delsarte \cite{Delsarte} showed that any clique $C$ of $G$ has the order at most $1-\frac{k}{\theta_{\min}(G)}$ and we call $C$ \emph{Delsarte} if its order is exactly $1-\frac{k}{\theta_{\min}(G)}$. We say $G$ is \emph{geometric} if there exists a set $\mathcal{C}$ of Delsarte cliques of $G$ such that every pair of adjacent vertices of $G$ lies in a unique $C\in\mathcal{C}$. In 1979, Neumaier \cite{Neumaier} showed that

\begin{thm}
For any integer $t\geq2$, there are only finitely many connected non-geometric strongly regular graphs with smallest eigenvalue at least $-t$.
\end{thm}

For geometric strongly regular graphs, we have the following result:
\begin{lem}
If $G$ is a geometric strongly regular graph, then $G$ is integrable.
\end{lem}
\begin{proof}
Without loss of generality, we may assume that $G$ has valency $k$ and smallest eigenvalue $\lambda_{\min}(G)$. (Note that $\lambda_{\min}(G)$ must be an integer.) Let $N$ be the vertex-clique incidence matrix with respect to the set $\mathcal{C}$ of Delsarte cliques, that is $N_{C,x}=\left\{
  \begin{array}{ll}
    1 &\text{ if $x \sim V(C)$} \\
    0 &\text{ otherwise} \\
  \end{array}
\right.$ for $x\in V(G)$ and $C\in\mathcal{C}$. Then $N^TN=A(G)-\theta_{\min}(G) I$ holds and this completes the proof of this lemma.
\end{proof}

Note that the complete multipartite graph $K_{n\times t}=K_{\tiny\underbrace{t,\ldots,t}_n}$, which is the strongly regular graph with parameters $(nt,(n-1)t,(n-2)t,(n-1)t)$ and smallest eigenvalue $-t$, is also integrable. The proof is as follows:

Let $\cup_{i=1}^n V_i$ be the partition of $K_{n\times t}$ into its color classes where $V_i=\{x^i_1,\ldots,x^i_t\}$, and let $\{\mathbf{e},\mathbf{e}_{i,l,j}\mid 1\leq i\leq n, 1\leq l<j\leq t\}$ be an integral orthonormal basis for $\mathbb{R}^{1+\frac{nt(t-1)}{2}}$. Define the $x_j^i$-column $N_{x^i_j}$ of $N$ as follows:
$$N_{x^i_j}:=\mathbf{e}-\sum_{1\leq l<j}\mathbf{e}_{i,l,j}+\sum_{j<l\leq t}\mathbf{e}_{i,j,l}.$$ We find that $N^TN=A(K_{n\times t})+tI$ and this completes the proof.

\subsection{Integrable strongly regular graphs with smallest eigenvalue at least $-2$}
Now let us look at the known results on the strongly regular graphs with smallest eigenvalue at least $-2$. In 1968, Seidel \cite{Seidel} showed that
\begin{thm}\label{srg-2}
 If $G$ is a connected strongly regular graph with smallest eigenvalue $-2$, then $G$ is a triangular graph $T(n)~(n\geq5)$, a lattice graph $L_2(n)~(n\geq3)$, a complete multipartite graph $K_{n\times2}~(n\geq2)$, or one of the graphs of Petersen, Clebsch, Schl\"{a}fli, Shrikhande, or Chang.
\end{thm}
From (\cite[Corollary 3.12.3]{drg}), we find that if $G$ is a connected strongly regular graph with $v$ vertices, valency $k$ and smallest eigenvalue $\theta_{\min}(G)>-2$, then either $G$ is complete or $G$ is a pentagon. Therefore, the strongly regular graphs with smallest eigenvalue at least $-2$, which are not integrable, are the Petersen graph, Clebsch graph, Schl\"{a}fli graph, Shrikhande graph and the three Chang graphs. Note that all of these seven strongly regular graphs are $2$-integrable, see \cite[Section 3.11]{drg}.

\subsection{Integrable strongly regular graphs with smallest eigenvalue at least $-3$}
In this subsection, we will show the following theorem.
\begin{thm}\label{mainintegrable}
Assume that $G$ is an integrable strongly regular graph with parameters $(v,k,\lambda,\mu)$ and smallest eigenvalue $\theta_{\min}(G)\in[-3,-2)$. If $\mu\geq10$, then $G$ is the strongly regular graph $K_{\frac{v}{3}\times 3}$.
\end{thm}
\begin{proof}
Let $\mu\geq10$. For convenience, we denote by $W_{x,y}:=\{w\mid w\not\sim x,w\not\sim y \text{ and }w\neq x,y\}$ the set of common nonadjacent vertices of $x$ and $y$. It is easy to see that
\begin{equation}\label{numberofv}
v=
\begin{cases}
2k-\lambda+|W_{x,y}| &\text{ if }x\sim y,\\
2+2k-\mu+|W_{x,y}| &\text{ if }x\not\sim y.
\end{cases}
\end{equation}
Since $G$ is integrable, there exists an integral matrix $N$ such that $$A(G)+3I=N^TN$$ holds. Note that $N_{i,j}\in\{1,-1,0\}$ for any $i,j$. For any vertex $x$, we denote by $N_x$ the $x$-column of $N$ and ${\rm supp}(N_x):=\{i\mid (N_x)_i\neq0\}$ the support of $N_x$.

For any two distinct vertices $x$ and $y$, we first loot at the case $|{\rm supp}(N_x)\cap {\rm supp}(N_y)|=3$. It follows immediately that $x$ and $y$ are adjacent and we may assume $N_x=\mathbf{e}_1+\mathbf{e}_2+\mathbf{e}_3$ and $N_y=\mathbf{e}_1+\mathbf{e}_2-\mathbf{e}_3.$ For any vertex $z\in V(G)-\{x,y\}$, if $z$ is adjacent to $x$, then $(N_z,N_x)=1$. As $0\leq(N_z,N_y)\leq1$, we find that $N_z=\mathbf{e}_1+\mathbf{e}_i+\mathbf{e}_j$ or $N_z=\mathbf{e}_2+\mathbf{e}_i+\mathbf{e}_j$ for some $3<i<j$, and thus $z\sim y$. This implies that $N_G(x)-\{y\}=N_G(y)-\{x\}$, that is, $\lambda=k-1$. Hence, we obtain that $G$ is the strongly regular graph $\frac{v}{k+1}K_{k+1}$ (see \cite[Lemma 10.1.1]{Ch&G}), which has smallest eigenvalue $-2$. This gives a contradiction. Therefore, we have
\begin{equation}\label{no3commonpositions}
|{\rm supp}(N_x)\cap {\rm supp}(N_y)|=1, \text{ if }x\sim y.
\end{equation}
From now on, we may assume $x\not\sim y$, and we will show that
\begin{equation}\label{commonposition}
|{\rm supp}(N_x)\cap {\rm supp}(N_y)|=2, \text{ if }x\not\sim y.
\end{equation}
Suppose not, then $|{\rm supp}(N_x)\cap {\rm supp}(N_y)|=0$ and we may assume $N_x=\mathbf{e}_1+\mathbf{e}_2+\mathbf{e}_3,~N_y=\mathbf{e}_4+\mathbf{e}_5+\mathbf{e}_6.$ Since $\mu\geq10$, there must exist two vertices $z,z^\prime\in N_G(x)\cap N_G(y)$ such that $|{\rm supp}(N_z)\cap {\rm supp}(N_{z^\prime})|=3$. This is not possible and (\ref{commonposition}) holds. Moreover, if there exists an integral unit vector $\mathbf{e}$ such that $(N_x,\mathbf{e})=(N_y,\mathbf{e})=1$, then we have
\begin{equation}\label{samesign}
(N_z,\mathbf{e})\geq0, \text{ for }z\in V(G).
\end{equation}
Otherwise let $N_x:=\mathbf{e}+\mathbf{e}'+\mathbf{e}^{\prime\prime}$ and $N_y:=\mathbf{e}-\mathbf{e}'+\mathbf{e}^{\prime\prime\prime}$ be the $x$-column and $y$-column of $N$. It is easy to see that, by (\ref{no3commonpositions}), there exist at least $\mu-1$ vertices in the set $N_G(x)\cap N_G(y)$ such that the columns of $N$ indexed by these vertices have inner product $1$ with $\mathbf{e}$, and $0$ with all $\mathbf{e}',~\mathbf{e}^{\prime\prime}$ and $\mathbf{e}^{\prime\prime\prime}$. Suppose that there exists a vertex $z$ such that $(N_z,\mathbf{e})=-1$, then $N_z=-\mathbf{e}+\mathbf{e}^{\prime\prime}+\mathbf{e}^{\prime\prime\prime}$. In this case, $N_z$ has inner product $-1$ with the column of $N$ indexed by some vertex in the set $N_G(x)\cap N_G(y)$. This is not possible and (\ref{samesign}) holds.

Given any two nonadjacent vertices $u$ and $v$, we may assume, by (\ref{commonposition}), that
$$N_u=\mathbf{e}_1+\mathbf{e}_2+\mathbf{e}_3,~N_v=\mathbf{e}_1-\mathbf{e}_2+\mathbf{e}_4.$$
Let us look at the set $W_{u,v}$. For any vertex $w\in W_{u,v}$, if $(N_w,\mathbf{e}_2)=1$, then we have, by (\ref{samesign}), that $(N_v,\mathbf{e}_2)=-1\geq0$ as $w\not\sim u$ and $(N_w,\mathbf{e}_2)=(N_u,\mathbf{e}_2)=1$. This is not possible. Thus $(N_w,\mathbf{e}_2)\neq1$. Similarly, $(N_w,-\mathbf{e}_2)\neq1$. By using (\ref{samesign}) again, we find that $(N_w,\mathbf{e}_1)\geq0$ and thus
$N_w=\mathbf{e}_1-\mathbf{e}_3-\mathbf{e}_4$. This implies that $$|W_{u,v}|\leq1.$$

If $|W_{u,v}|=0$, then $v=2+2k-\mu$ by (\ref{numberofv}). For the number $k-\mu$, which is the cardinality of the set $N_G(u)-N_G(u)\cap N_G(v)$, we have $k-\mu\geq1$ otherwise $G=K_{\frac{v}{2}\times 2}$ with $\theta_{\min}(G)=-1$. For any vertex $u'\in N_G(u)-N_G(u)\cap N_G(v)$, we have $(N_{u'},-\mathbf{e}_2)\neq1$. Otherwise $(N_u,-\mathbf{e}_2)=-1\geq0$ by (\ref{samesign}), as $u'\not\sim v$ and $(N_{u'},-\mathbf{e}_2)=(N_v,-\mathbf{e}_2)=1$. This is not possible. Note that $|{\rm supp}(N_{u'})\cap {\rm supp}(N_u)|=1$, $|{\rm supp}(N_{u'})\cap {\rm supp}(N_v)|=2$ and $(N_{u'},\mathbf{e}_1)\geq0$ by (\ref{no3commonpositions})-(\ref{samesign}). Thus,
$$N_{u'}=\mathbf{e}_1-\mathbf{e}_4+\mathbf{e}_i \text{ or }N_{u'}=\mathbf{e}_2+\mathbf{e}_4+\mathbf{e}_i \text{ for some }i\geq5.$$
If $k-\mu=|N_G(u)-N_G(u)\cap N_G(v)|=|\{u',u^{\prime\prime}\}|=2$, then we may assume $N_{u'}=\mathbf{e}_1-\mathbf{e}_4+\mathbf{e}_i$ and $N_{u^{\prime\prime}}=\mathbf{e}_2+\mathbf{e}_4+\mathbf{e}_i$ for some $i\geq5$. As $v\not\sim u^{\prime\prime}$ and $(N_v,\mathbf{e}_4)=(N_{u^{\prime\prime}},\mathbf{e}_4)=1$, we have, by (\ref{samesign}), that $(N_{u^{\prime}},\mathbf{e}_4)=-1\geq0$. This gives the contradiction. Thus, $k-\mu=1$ and $v=3+k$. We conclude that the complement $\overline{G}$ of $G$ is $2$-regular and any two adjacent vertices of $\overline{G}$ have no common neighbors. Since $\overline{G}$ is a strongly regular graph with at least $12$ vertices, we find that this is not possible. Therefore, $|W_{u,v}|=1$.

Assume that $W_{u,v}=\{w\}$, then $N_w=\mathbf{e}_1-\mathbf{e}_3-\mathbf{e}_4$ and it is easy to check that $N_{G}(u)=N_{G}(v)=N_{G}(w)$, that is, $\mu=k$ and $v=2+k+1$ by (\ref{numberofv}). We conclude that the complement $\overline{G}$ of $G$ is $2$-regular and any two adjacent vertices of $\overline{G}$ have exactly one common neighbor, that is, $\overline{G}$ is a disjoint union of $K_3$. This implies that $G$ is the graph $K_{\frac{v}{3}\times 3}$.

This completes the proof.
\end{proof}

\section{The complement of the Sims-Gewirtz graph}
In 1969, A. Gewirtz \cite{gewirtz} showed that the strongly regular graph with parameters $(56,10,0,2)$ exists and is unique, named the Sims-Gewirtz graph. Thus, the strongly regular graph, by (\ref{compleparameter}), with parameters $(56,45,36,36)$ is also unique, which is the complement of the Sims-Gewirtz graph, and its spectrum is $\{45^{(1)},3^{(20)},-3^{(35)}\}$. By Theorem \ref{mainintegrable}, we easily find that the latter strongly regular graph is not integrable. Note that
\begin{lem}(\cite[Theorem 4]{2quasidesign})\label{lemforgewirtz}
The block graph of the quasi-symmetric $2$-$(21,6,4)$ design with intersection numbers $l_{1}=0$ and $l_{2}=2$ is strongly regular with parameters $(56,45,36,36)$.
\end{lem}
Thus, we have the following result.
\begin{thm}\label{complementofsg}
The strongly regular graph $\overline{SG}$, which is the complement of the Sims-Gewirtz graph, with parameters $(56,45,36,36)$ and smallest eigenvalue $-3$ is $2$-integrable.
\end{thm}
\begin{proof}By using Lemma \ref{lemforgewirtz} and the fact that the strongly regular graph with parameters $(56,45,36,36)$ is unique, we conclude that $\overline{SG}$ is the block graph of the quasi-symmetric $2$-$(21,6,4)$ design with intersection numbers $l_{1}=0$ and $l_{2}=2$. Denote by $\mathcal{I}$ the incidence matrix of this design. We find that $2A(\overline{SG})+6I=\mathcal{I}^T\mathcal{I}$, that is,
$$A(\overline{SG})+3I=\frac{1}{2}\mathcal{I}^T\mathcal{I}.$$
This completes the proof.
\end{proof}

\section{The Hoffman-Singleton graph}
The Hoffman-Singleton graph, constructed by A.J. Hoffman and R.R. Singleton in 1960 \cite{AJR}, is the unique strongly regular graph with parameters $(50,7,0,1)$ and spectrum  $\{7^{(1)},2^{(28)},(-3)^{(21)}\}$. For the uniqueness, see \cite{HSunique}. Here we give a construction of the Hoffman-Singleton graph, which is introduced by N.Robertson, see \cite[p. 391]{drg}.

Let $P^i$ and $Q^i$  be the graphs in Figure \ref{figure},
\begin{center}
\begin{figure}[H]
\includegraphics[scale=0.7]{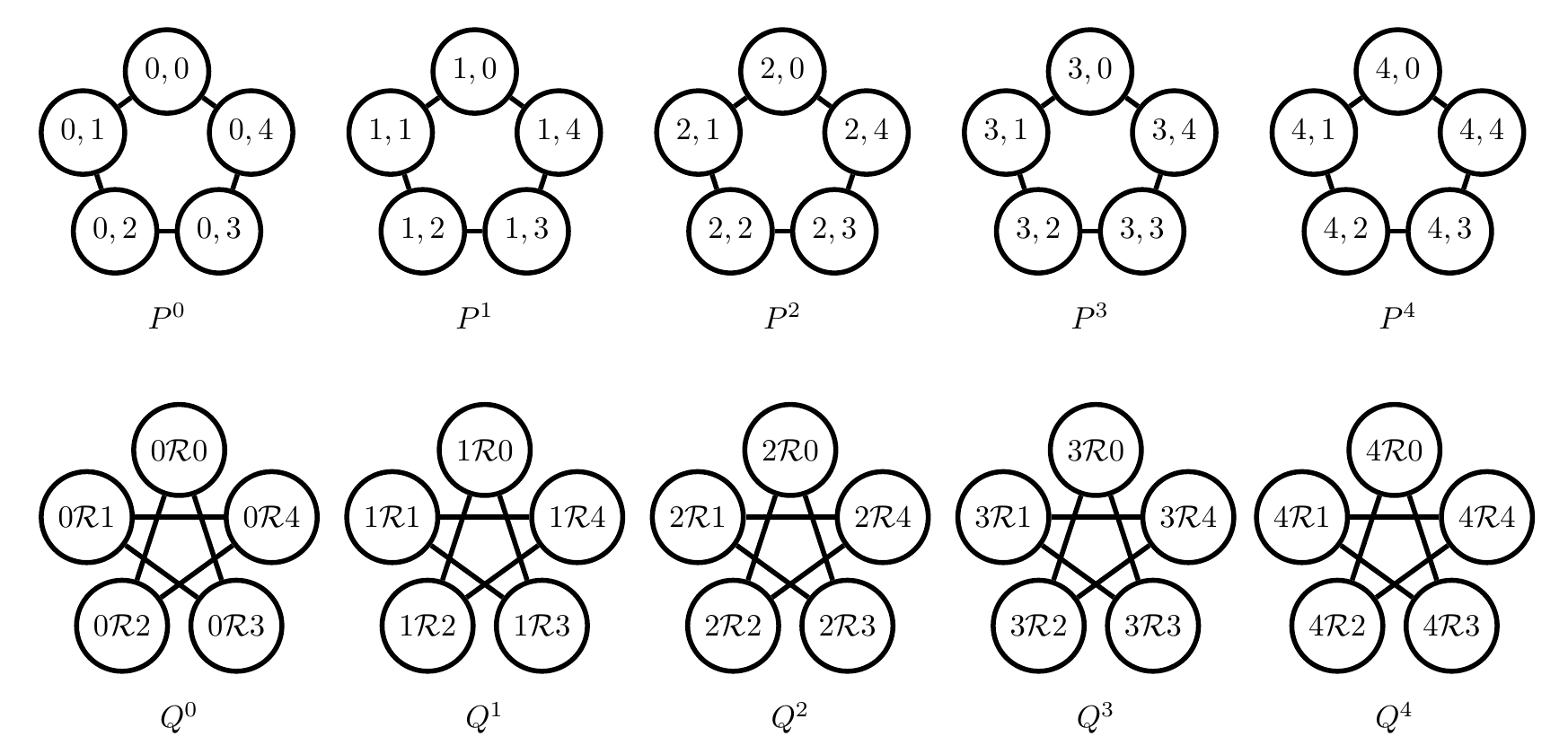}
{\caption{}\label{figure}}
\end{figure}
\end{center}
where $V(P^i)=\{(i,0),\ldots,(i,4)\}$ and $V(Q^i)=\{i\mathcal{R}0,\ldots,i\mathcal{R}4\}$. Now join vertex $i\mathcal{R}l$ of $Q^i$ to vertex $(j,i\cdot j+l)$ of $P^j$ for all $0\leq i,j,l\leq4$. Here $i\cdot j+l$ is calculated modulo $5$. This yields the Hoffman-Singleton graph.

For the pentagons in a Hoffman-Singleton graph, C. Fan and A.J. Schwenk \cite{structure} showed the following.
\begin{pro}\label{StructureHS}
Suppose that $HoSi$ is a Hoffman-Singleton graph with a pentagon $H_0$ as a subgraph. Denote by $N_i(H_0)$ the set of vertices not in $H_0$ each of which has distance $i$ from some vertex of $H_0$, where $i=1,2$. Then the subgraph induced by $N_1(H_0)$ is a disjoint union of $5$ pentagons $H^0,\ldots,H^4$ and the subgraph induced by $N_2(H_0)$ is a disjoint union of $4$ pentagons $H_1,\ldots,H_4$. Moreover, the induced subgraph on $V(H_i)\cup V(H^j)$ is a Petersen graph for $0\leq i,j\leq4$.
\end{pro}

In \cite{AutHS}, P.R. Hafner showed that

\begin{pro}\label{AutHS}
In a Hoffman-Singleton graph, there are $1260$ pentagons and $126$ sets of $10$ disjoint pentagons. The automorphism group of the Hoffman-Singleton graph acts transitively on the set of all sets of $10$ disjoint pentagons.
\end{pro}

In this section, we will show the following result.
\begin{thm}\label{HS}
The Hoffman-Singleton graph $HoSi$ with parameters $(50,7,0,1)$ and smallest eigenvalue $-3$ is not $2$-integrable.
\end{thm} \begin{proof}
Suppose that $HoSi$ is $2$-integrable. Then there exists an integral matrix $N$ such that
$2A(HoSi)+6I=N^TN$ holds. Note that for any entry $N_{i,j}$ of $N$, $N_{i,j}\in\{\pm1,\pm2,0\}$. Without loss of generality, we may assume that for any row $\mathbf{r}$ of $N$, $|{\rm supp}(\mathbf{r})|>0$ and $\sigma^{\mathbf{r}}\geq0$.

For any $x\in V(HoSi)$, let $\pi(x):=\{\{x\},N_1(x),N_2(x)\}$ be an equitable partition of $HoSi$. From (\ref{quotientmatrix}), we find the quotient matrix $Q_{\pi(x)}$ of $HoSi$ relative to $\pi(x)$ is as follows:
$$Q_{\pi(x)}=\begin{pmatrix}
0 & 7 & 0 \\
1 & 0 & 6\\
0 & 1 & 6
\end{pmatrix}.$$
The vector $\mathbf{u}:=(21,-9,1)^T$ is an eigenvector of $Q_{\pi(x)}$ with eigenvalue $-3$. We have, by (\ref{stronglysupport}),
\begin{equation}\label{hssupposrteq}
21\gamma_x^\mathbf{r}-9\delta_x^\mathbf{r}+\zeta_x^\mathbf{r}=0
\end{equation}
holds, for any row $\mathbf{r}$ of $N$.

First, we claim that each entry of $N$ equals $1,~-1$ or $0$. Suppose not, then there exists a row $\mathbf{r}^\prime$ of $N$ and a vertex $x$ of $HoSi$ such that $|\mathbf{r}^\prime_x|=2$. It is straightforward to check that there exists no vertex $y$ satisfying $|\mathbf{r}^\prime_y|=2$ except the vertex $x$. From Proposition \ref{coclipro}, we obtain that the graph $HoSi(\mathbf{r}^\prime)$ has no $4$-coclique. Note that $HoSi(\mathbf{r}^\prime)$ also has no $3$-clique, as the graph $HoSi$ does not contain $3$-clique. Thus, we have
\begin{equation}\label{hs2supportbound}
|\gamma_x^{\mathbf{r}^\prime}|-1+|\delta_x^{\mathbf{r}^\prime}|+|\zeta_x^{\mathbf{r}^\prime}|\leq|{\rm supp}(\mathbf{r}^\prime)|<R(3,4)=9.
\end{equation}
But we can not find the possible solution for $(\gamma_x^{\mathbf{r}^\prime},\delta_x^{\mathbf{r}^\prime},\zeta_x^{\mathbf{r}^\prime})$ which satisfies both (\ref{hssupposrteq}) and (\ref{hs2supportbound}). This implies that the claim holds.

Now for any row $\mathbf{r}$ of $N$, we have that, from Proposition \ref{coclipro}, the graph $HoSi(\mathbf{r})$ has no $7$-coclique. Note that $HoSi(\mathbf{r})$ also has no $3$-clique, as the graph $HoSi$ does not contain $3$-clique. Thus, we obtain
\begin{equation}\label{hssupportsize}
|\gamma_x^\mathbf{r}|+|\delta_x^\mathbf{r}|+|\zeta_x^\mathbf{r}|\leq|{\rm supp}(\mathbf{r})|<R(3,7)=23,
\end{equation}
\begin{equation}\label{hssupportzeta}
|\zeta_x^\mathbf{r}|<R(3,6)=18.
\end{equation}
Thus, we have
\begin{equation}\label{hspossibleparameters}
(\gamma_x^\mathbf{r},\delta_x^\mathbf{r},\zeta_x^\mathbf{r},\sigma^\mathbf{r})\in\{(-1,-1,12,10),(1,3,6,10),(1,4,15,20),(\pm1,\pm2,\mp3,0)\}
\end{equation}
for any $x$ with $\mathbf{r}_x\neq0$.

First we show $\sigma^\mathbf{r}\neq20$. Otherwise, for any $x$ with $\mathbf{r}_x\neq0$, we have $\gamma_x^\mathbf{r}=\mathbf{r}_x=1$ and $\delta_x^\mathbf{r}=4$. From (\ref{normofr}), we also have $4|{\rm supp}(\mathbf{r})|\geq\frac{|{\rm supp}(\mathbf{r})|^2}{2}-3|{\rm supp}(\mathbf{r})|$, that is, $|{\rm supp}(\mathbf{r})|\leq14$. But $|{\rm supp}(\mathbf{r})|\geq1+4+15=20$. This gives a contradiction.

If $\sigma^\mathbf{r}=10$, then $\delta_x^\mathbf{r}=3$ when $\gamma_x^\mathbf{r}=\mathbf{r}_x=1$ and $\delta_x^\mathbf{r}=-1$ when $\gamma_x^\mathbf{r}=\mathbf{r}_x=-1$. By using (\ref{normofr}), we have $3\cdot\frac{|{\rm supp}(\mathbf{r})|+10}{2}+\frac{|{\rm supp(\mathbf{r}})|-10}{2}\geq\frac{|{\rm supp}(\mathbf{r})|^2}{2}-3|{\rm supp}(\mathbf{r})|$ and thus $|{\rm supp}(\mathbf{r})|\leq10$. This implies that $|{\rm supp}(\mathbf{r})|=10$ and for any $x$ with $\mathbf{r}_x\neq0$, $(\gamma_x^\mathbf{r},\delta_x^\mathbf{r},\zeta_x^\mathbf{r})=(1,3,6)$, that is, the induced subgraph $HoSi(\mathbf{r})$ is the Petersen graph.

If $\sigma^\mathbf{r}=0$, then $\delta_x^\mathbf{r}=2$ when $\gamma_x^\mathbf{r}=\mathbf{r}_x=1$ and $\delta_x^\mathbf{r}=-2$ when $\gamma_x^\mathbf{r}=\mathbf{r}_x=-1$, that is, for any vertex $x$ with $\mathbf{r}_x\neq0$, there exist at least two vertices $y_1$ and $y_2$ such that $y_i\sim x$ and $\mathbf{r}_{y_i}=\mathbf{r}_{x}$. Considering that $HoSi$ has no $3$- and $4$-cycle, we find that $|{\rm supp}(\mathbf{r})|=|V(HoSi(\mathbf{r}))|\geq10$. Note that in this case, we also have
$$|{\rm supp}(\mathbf{r})|^2=(\mathbf{r}\mathbf{r}^T)^2\leq \mathbf{r}N^TN\mathbf{r}^T=2\mathbf{r}A(HoSi)\mathbf{r}^T+6\mathbf{r}\mathbf{r}^T=2\cdot2\cdot|{\rm supp}(\mathbf{r})|+6|{\rm supp}(\mathbf{r})|.$$
Thus $|{\rm supp}(\mathbf{r})|=10$ and the induced subgraph $HoSi(\mathbf{r})$ is a disjoint union of $2$ pentagons. Moreover, the row $\mathbf{r}$ is orthogonal to all of the other rows of $N$ except itself.

Now let us look at the matrix $N$. It is easy to see that $N$ has exactly $30$ rows, since $N$ has $300$ $(=6\times 50)$ non-zero entries and each row $\mathbf{r}$ of $N$ has $10$ non-zero entries. Denote by $\mathbf{r}^0,\ldots,\mathbf{r}^{29}$ all of the rows of $N$. For convenience, we may assume $\sigma^{\mathbf{r}^i}=10$ for $0\leq i\leq s-1$ and $\sigma^{\mathbf{r}^i}=0$ for $s\leq i\leq 29$. Considering that ${\rm rank}(N)={\rm rank}(A(HoSi)+3I)=50-21=29$ and $\mathbf{r}^j(\mathbf{r}^i)^T=0$ for $s\leq j\leq 29$ and $j\neq i$, there exist constants $f_1,\ldots,f_{s}$ such that the following holds.

\begin{equation}\label{lineardependent}
\sum_{i=0}^{s-1}f_i\mathbf{r}^i=\mathbf{0}, \text{ where } \sum_{i=0}^{s-1}f_i^2\neq0
\end{equation}

Let $H$ be the graph with vertex set $\{\mathbf{r}^0,\ldots,\mathbf{r}^{s-1}\}$, where $\mathbf{r}^i$ and $\mathbf{r}^j$ are adjacent if and only if there exists a vertex $y\in V(G)$ such that $\mathbf{r}^i_y=\mathbf{r}^j_y=1$. For any vertex $x$, assume $l_0(x):=|\{\mathbf{r}^i\mid \mathbf{r}^i_x\neq0, 0\leq i\leq s-1\}|$ and $l_1(x):=|\{\mathbf{r}^i\mid \mathbf{r}^i_x\neq0, s\leq i\leq 29\}|$. We have
\begin{equation*}
\begin{split}
&\sum_{0\leq i\leq 29,\mathbf{r}^i_x\neq0}|\gamma_x^{\mathbf{r}^i}|=l_0(x)+l_1(x)=6,\\
&\sum_{0\leq i\leq 29,\mathbf{r}^i_x\neq0}|\delta_x^{\mathbf{r}^i}|=3l_0(x)+2l_1(x)=2\times 7=14.
\end{split}
\end{equation*}
This shows
\begin{equation}\label{vertexinpetersen}
l_0(x)=2.
\end{equation}
By using (\ref{vertexinpetersen}), we obtain
$$|\{(x,\mathbf{r}^i)\mid \mathbf{r}^i_x\neq0,0\leq i\leq s-1\}|=2\cdot 50=10s.$$
Thus $|V(H)|=s=10$.

Now we figure out the minimal valency of $H$. For any vertex $\mathbf{r}^i$ of $H$, let $\pi:=\{V(HoSi(\mathbf{r}^i)),\\V(HoSi)-V(HoSi(\mathbf{r}^i))\}$ be a partition of the vertices of $HoSi$. The quotient matrix $Q_{\pi}$ of $HoSi$ relative to $\pi$ is $$Q_{\pi}=\left(\begin{array}{cc}3 & 4 \\1 & 6 \\\end{array}\right)$$ with $7$ and $2$ as eigenvalues. Note that $\pi$ is equitable by Lemma \ref{quotienteigenvalues1} and the vector $\mathbf{w}:=(4,-1)^T$ is an eigenvector of $Q_{\pi}$ with $2$ as eigenvalue. Define $\widetilde{\mathbf{w}}$ to be a vector satisfying $\widetilde{\mathbf{w}}_x=4$ if $x\in V(HoSi(\mathbf{r}^i))$ and $\widetilde{\mathbf{w}}_x=-1$ if $x\in V(HoSi)-V(HoSi(\mathbf{r}^i))$. Then $A(HoSi)\widetilde{\mathbf{w}}=2\widetilde{\mathbf{w}}$ and $\widetilde{\mathbf{w}}^T(2A(HoSi)+6I)\widetilde{\mathbf{w}}=10\widetilde{\mathbf{w}}^T\widetilde{\mathbf{w}}=2000$.
Note that for any $\mathbf{r}^j\in V(H)$, $\mathbf{r}^j$ is a $01$-vector. Hence, if $\mathbf{r}^i(\mathbf{r}^j)^T=t$, then $(\mathbf{r}^j\widetilde{\mathbf{w}},\mathbf{r}^j\widetilde{\mathbf{w}})=(5t-10)^2$. Define $p^i_t:=|\{j\mid j\neq i,\mathbf{r}^j\in V(H),\mathbf{r}^i(\mathbf{r}^j)^T=t\}|$, where $0\leq t\leq 10$. Thus,
\begin{equation}\label{intersectionnumber}
\sum_{t=1}^{10}p^i_t=9,~\sum_{t=1}^{10}p^i_t(5t-10)^2\leq\widetilde{\mathbf{w}}^TN^TN\widetilde{\mathbf{w}}-(\mathbf{r}^i\widetilde{\mathbf{w}},\mathbf{r}^i\widetilde{\mathbf{w}})=\widetilde{\mathbf{w}}^T(2A(HoSi)+6I)\widetilde{\mathbf{w}}-1600=400.
\end{equation}
From (\ref{intersectionnumber}), we find $p^i_0\leq4$. This shows that for any vertex $\mathbf{r}^i$ of $H$, it has valency at least $5$.

Since $\mathbf{r}^i$ is a $01$-vector for any $0\leq i\leq s-1$, we infer that, from (\ref{lineardependent}) and (\ref{vertexinpetersen}), any two distinct vertices $\mathbf{r}^i$ and $\mathbf{r}^j$ of $H$ are adjacent if and only if $f_i=-f_j$. As the minimal valency of $H$ is at least $5$ and  $\sum_{i=0}^{s-1}f_i^2\neq0$, we conclude that $H$ is the bipartite graph $K_{5,5}$ and thus
 \begin{equation}\label{intersectionnumber1}
p^i_0=4,~p^i_2=5,\text{ and }p^i_t=0 \text{ for }t\neq0,2.
\end{equation}

 Assume $\{\mathbf{r}_0,\ldots,\mathbf{r}_4\}$ and $\{\mathbf{r}_5,\ldots\mathbf{r}_9\}$ are the color classes of $H$. The induced subgraphs $HoSi(\mathbf{r}^0)$, $\ldots$, $HoSi(\mathbf{r}^9)$ have the following properties.
\begin{enumerate}
 \item $HoSi(\mathbf{r}^i)$ is a Petersen graph for any $0\leq i\leq 9$;
 \item $|V(HoSi(\mathbf{r}^i))\cap V(HoSi(\mathbf{r}^j))|=0$ for $0\leq i<j\leq4$;
 \item $|V(HoSi(\mathbf{r}^{i}))\cap V(HoSi(\mathbf{r}^{j}))|=0$ for $5\leq i<j\leq9$;
 \item $|V(HoSi(\mathbf{r}^i))\cap V(HoSi(\mathbf{r}^j))|=2$ for $0\leq i\leq 4$ and $5\leq j\leq 9$.
 \end{enumerate}

Without loss of generality, we may assume, by Proposition \ref{AutHS} and Proposition \ref{StructureHS},  $V(HoSi(\mathbf{r}^i))=V(P^i)\cup V(Q^i)$ for $i=0,\ldots,4$, where $P^i$ and $Q^i$ are the graphs in Figure \ref{figure}. But now we can not find the Petersen graph $HoSi(\mathbf{r}^5)$ satisfying $|V(HoSi(\mathbf{r}^5))\cap V(HoSi(\mathbf{r}^i))|=2$ for $0\leq i\leq 4$.

This shows that the graph $HoSi$ is not $2$-integrable and we complete the proof.
\end{proof}

\section{The complement of $GQ(3,9)$}

A \emph{generalized quadrangle} of order $(s, t)$ is an incidence structure of points and lines with the
properties that
\begin{enumerate}
\item every point lies on $t+1$ lines and any two points are on at most one line;
\item every line contains $s+1$ points;
\item for any point $P$ and line $L$ which are not incident, there is a unique point on $L$ collinear with $P$.
\end{enumerate}
\indent The \emph{point graph} of a generalized quadrangle, also denoted by $GQ(s, t)$, is the graph with the points of the quadrangle as its vertices, with two points adjacent if and only if they are collinear. Note that the graph $GQ(s,t)$ is a strongly regular graph with parameters
\begin{equation}\label{gqparameter}
((st+1)(s+1),s(t+1),s-1,t+1)
\end{equation} (see \cite[Lemma 10.8.1]{Ch&G}).

It is well-known that the strongly regular graph $GQ(q, q^{2})$ exists when $q$ is a prime power, and is unique for $q = 2, 3$ (see \cite{finitegeneralizedquadrangles}). The complement of $GQ(q, q^{2})$, denoted by $\overline{GQ}(q,q^{2})$, is a strongly regular graph with parameters, by (\ref{compleparameter}) and (\ref{gqparameter}), as follows:
$$((q + 1)(q^{3} + 1), q^{4}, q(q - 1)(q^{2} + 1), (q - 1)q^{3}).$$
In particular $GQ(3,9)$ is the unique strongly regular graph with parameters $(112,30,2,10)$ and spectrum $\{30^{(1)},2^{(90)},-10^{(21)}\}$, and its first subconstituent is a disjoint union of $10$ $3$-cliques and second subconstituent is the unique strongly regular graph with parameters $(81,20,1,6)$, called the Brouwer-Haemers graph (see \cite{brouwer&heamers}). Therefore, we have
\begin{pro}\label{gq39subgraph}
\begin{enumerate}
\item The graph $\overline{GQ}(3,9)$ is the unique strongly regular graph with parameters $(112,81,60,54)$ and spectrum $\{81^{(1)},9^{(21)},-3^{(90)}\}$.
\item The first subconstituent of $\overline{GQ}(3,9)$ is the strongly regular graph with parameters $(81,60,45,42)$, which is the complement of the Brouwer-Haemers graph.
\item The maximal order of cocliques in $\overline{GQ}(3,9)$ is $4$ and the maximal order of cliques in $\overline{GQ}(3,9)$ is $16$.
\end{enumerate}
\end{pro}
In this section, we will show the following result.
\begin{thm}\label{gq39}
The strongly regular graph $\overline{GQ}(3,9)$, which is the complement of the point graph of the generalized quadrangle of order $(3,9)$, with parameters $(112,81,60,54)$ and smallest eigenvalue $-3$ is not $2$-integrable.
\end{thm}

\begin{proof}
Suppose that the graph $\overline{GQ}(3,9)$ is $2$-integrable, then there exists an integral matrix $N$ such that $2A(\overline{GQ}(3,9))+6I=N^{T}N$ holds. Note that for any entry $N_{i,j}$, $N_{i,j}\in\{\pm1,\pm2,0\}$. Without loss of generality, we may assume that for any row $\mathbf{r}$ of $N$, $|{\rm supp}(\mathbf{r})|>0$ and $\sigma^{\mathbf{r}}\geq0$. Choose $\mathbf{r}^\prime$ to be a row of $N$ whose support has the minimum size.

From Proposition \ref{gq39subgraph} \rm{(iii)}, we find that there exists a $4$-coclique $\overline{C}$ in $\overline{GQ}(3,9)$. Let $\pi=\{V(\overline{C}),V(G)-V(\overline{C})\}$ be a partition of $\overline{GQ}(3,9)$. The quotient matrix $Q_{\pi}$ of $\overline{GQ}(3,9)$ relative to $\pi$ is $\left(\begin{array}{cc}0 & 81\\3 & 78 \\\end{array}\right)$ with $81$ and $-3$ as eigenvalues, where $(27,-1)^T$ is an eigenvector of $Q_{\pi}$ with eigenvalue $-3$. Thus, $\pi$ is an equitable partition by Lemma \ref{quotienteigenvalues1} and $27\sum_{x\in V(\overline{C})}\mathbf{r}^\prime_x-\sum_{x\in V(\overline{GQ}(3,9))-V(\overline{C})}\mathbf{r}^\prime_x=0$ holds by Propostition \ref{basicpro} \rm{(ii)}. This implies
\begin{equation}\label{gqcoclique}
28\mid\sigma^{\mathbf{r}^\prime}, \text{ where }\sigma^{\mathbf{r}^\prime}=\sum_{x\in V(\overline{GQ}(3,9))}\mathbf{r}^\prime_x.
\end{equation}

For any $x\in V(\overline{GQ}(3,9))$, the quotient matrix $Q_{\pi(x)}$, by (\ref{quotientmatrix}), of $\overline{GQ}(3,9)$ relative to the equitable partition $\pi(x):=\{\{x\},N_1(x),N_2(x)\}$ is as follows:
$$Q_{\pi(x)}=\begin{pmatrix}
0 & 81 & 0 \\
1 & 60 & 20\\
0 & 54 & 27
\end{pmatrix},$$
and the vector $\mathbf{u}:=(135,-5,9)^T$ is an eigenvector of $Q_{\pi(x)}$ with eigenvalue $-3$. We have, by (\ref{stronglysupport}),
\begin{equation}\label{gqpartition}
135\gamma_x^{\mathbf{r}^\prime}-5\delta_x^{\mathbf{r}^\prime}+9\zeta_x^{\mathbf{r}^\prime}=0.
\end{equation}

If there exists a vertex $x$ such that $\mathbf{r}^\prime_x=\gamma_x^{\mathbf{r}^\prime}=\pm2$, then for any vertex $y\neq x$, $\mathbf{r}^\prime_y\in\{\pm1,0\}$. In this case, \begin{equation}\label{gqfirstboud}
|{\rm supp}(\mathbf{r}^\prime)|\geq|\gamma_x^{\mathbf{r}^\prime}|-1+|\delta_x^{\mathbf{r}^\prime}|+|\zeta_x^{\mathbf{r}^\prime}|\geq31
\end{equation}
by (\ref{gqcoclique}) and (\ref{gqpartition}).

Now we may assume $\mathbf{r}^\prime_x\in\{\pm1,0\}$ for all $x\in V(\overline{GQ}(3,9))$. We will show that in this case $|{\rm supp}(\mathbf{r}^\prime)|\geq31$ also holds. By using (\ref{gqcoclique}) and (\ref{gqpartition}) again, only the following cases should be discussed.
\begin{equation}\label{gqpossibleparameters}
(\gamma_y^{\mathbf{r}^\prime},\delta_y^{\mathbf{r}^\prime},\zeta_y^{\mathbf{r}^\prime},\sigma^{\mathbf{r}^\prime})\in\{(1,27,0,28),(-1,9,20,28),(\pm1,\pm9,\mp10,0)\}
\end{equation}
with $\mathbf{r}^\prime_y\neq0$.

We claim that if $\delta_y^{\mathbf{r}^\prime}=\pm9$, there exist $17$ vertices $w_1,\ldots,w_{17}$ such that $w_i\sim y$ and $\mathbf{r}^\prime_{w_i}\neq0$. Define $\mathbf{r}^\prime(y)$ be the vector obtained from $\mathbf{r}^\prime$ be removing the coordinates indexed by $V(\overline{GQ}(3,9))-\{y\}-N_2(y)$. It is sufficient to show
that $|{\rm supp}(\mathbf{r}^\prime(y))|\geq17$. Now let us look at the first subconstituent $G_y$ of $\overline{GQ}(3,9)$, that is, the subgraph induced by the vertex set $N_1(y)$. From  Proposition \ref{gq39subgraph} {\rm (ii)}, we find that $G_y$ is a strongly regular graph with parameters $(81,60,45,42)$ and smallest eigenvalue $-3$. The quotient matrix $Q_{\pi^\prime(z)}$ of $G_y$ relative to the partition $\pi^\prime(z)=\{\{z\},N_1(z),N_2(z)\}$ of the vertices of $G_y$ is as follows:
$$Q_{\pi^\prime(z)}=\begin{pmatrix}
0 & 60 & 0 \\
1 & 45 & 14\\
0 & 42 & 18
\end{pmatrix},$$
and the vector $(20,-1,2)^T$ is an eigenvector of $Q_{\pi^\prime(z)}$ with eigenvalue $-3$. It is not hard to check that $20\gamma_z^{\mathbf{r}^\prime(y)}-\delta_z^{\mathbf{r}^\prime(y)}+2\zeta_z^{\mathbf{r}^\prime(y)}=0$ holds. If $\delta_y^{\mathbf{r}^\prime}=\sigma^{\mathbf{r}^\prime(y)}=\pm9$, we find $(\gamma_z^{\mathbf{r}^\prime(y)},\delta_z^{\mathbf{r}^\prime(y)},\zeta_z^{\mathbf{r}^\prime(y)})=(\pm1,\pm12,\mp4)$ when $\mathbf{r}^\prime(y)_z\neq0$ and thus $|{\rm supp}(\mathbf{r}^\prime(y))|\geq17$.

Now we claim that $\sigma^{\mathbf{r}^\prime}\neq0$. Otherwise $(\gamma_y^{\mathbf{r}^\prime},\delta_y^{\mathbf{r}^\prime},\zeta_y^{\mathbf{r}^\prime})=(\pm1,\pm9,\mp10)$ for any $\mathbf{r}^\prime_y\neq0$ and $|{\rm supp}(\mathbf{r}^\prime)|\geq 1+|{\rm supp}(\mathbf{r}^\prime(y))|+10\geq1+17+10=28$. We also have, by (\ref{normofr}), $9|{\rm supp}(\mathbf{r}^\prime)|\geq\frac{1}{2}|{\rm supp}(\mathbf{r}^\prime)|^2-3|{\rm supp}(\mathbf{r}^\prime)|$ and this is not possible.

If $\sigma^{\mathbf{r}^\prime}=28$, there must be a vertex $z^-$ such that
$(\gamma_{z^-}^{\mathbf{r}^\prime},\delta_{z^-}^{\mathbf{r}^\prime},\zeta_{z^-}^{\mathbf{r}^\prime})=(-1,9,20)$ and thus $|{\rm supp}(\mathbf{r}^\prime)|\geq 1+17+20\geq31$, otherwise the induced subgraph $\overline{GQ}(3,9)(\mathbf{r}^\prime)$ is a $28$-clique and this contradicts Proposition \ref{gq39subgraph} \rm{(iii)}.

But we have, by Proposition \ref{basicpro} \rm{(i)}, $|{\rm supp}(\mathbf{r}^\prime)|\leq\frac{2\cdot3\cdot112}{{\rm rank}(N)}=\frac{672}{{\rm rank}(A(\overline{GQ}(3,9))+3I)}=\frac{672}{112-90}<31.$ This shows that the graph $\overline{GQ}(3,9)$ is not $2$-integrable and the theorem holds.
\end{proof}

\section{The complement of the McLaughlin graph}
In 1975, J.M. Goethals and J.J. Seidel \cite{mc} showed that the McLaughlin graph is the unique strongly regular graph with parameters $(275,112,30,56)$ and spectrum $\{112^{(1)},2^{(252)},-28^{(22)}\}$. For its first and second subconstituents, they are the unique strongly regular graphs with parameters $(112,30,2,10)$ and $(162,56,10,24)$, respectively (see \cite{PJcemo} and \cite{van lint&brouwer}). Therefore we have
\begin{pro}\label{mcinformation}
\begin{enumerate}
\item The complement of the McLaughlin graph is the unique strongly regular graph with parameters $(275,162,105,81)$ and spectrum $\{162^{(1)},27^{(22)},-3^{(252)}\}$.
\item The second subconstituent of the complement of the McLaughlin graph is the strongly regular graph $\overline{GQ}(3,9)$.
\end{enumerate}
\end{pro}

Note that the McLaughlin graph can be constructed from the Steiner system $S(4,7,23)$ (cf. \cite[Chapter 11]{drg}). Thus, we have the following construction for the complement of McLaughlin graph.
\begin{pro}\label{constructcommcl}
Let $(\Omega,\mathcal{B})$ be the Steiner system $(4,7,23)$ with $\Omega=\{0,1,\ldots,22\}$. The complement of McLaughlin graph can be constructed as a graph with vertex set $\{1,\ldots,22\}\cup\mathcal{B}_1\cup\mathcal{B}_2$, where $\mathcal{B}_1$ is the set of $77$ blocks in $S(4,7,23)$ which contains $0$ and $\mathcal{B}_2$ is the set of $176$ blocks in $S(4,7,23)$ which does not contain $0$. The edges in this graph are defined as follows:
\begin{enumerate}
\item vertices in $\{1,\ldots,22\}$ are pairwise adjacent;
\item a vertex $x$ in $\{1,\ldots,22\}$ is adjacent to a block $B\in\mathcal{B}_1$ if and only if $x\in B$;
\item a vertex $x$ in $\{1,\ldots,22\}$ is adjacent to a block $B\in\mathcal{B}_2$ if and only if $x\not\in B$;
\item two blocks $B_1,B_2\in\mathcal{B}_i$ are adjacent if and only if $|B_1\cap B_2|=3$, for $i=1,2$;
\item two blocks $B_1\in\mathcal{B}_1$ and $B_2\in\mathcal{B}_2$ are adjacent if and only if $|B_1\cap B_2|=1$.
\end{enumerate}
\end{pro}

In this section, we will show that
\begin{thm}\label{mc}
The strongly regular graph $\overline{McL}$, which is the complement of the McLaughlin graph, with parameters $(275,162,105,81)$ and smallest eigenvalue $-3$ is $4$-integrable.
\end{thm}
\begin{proof}In order to do so, we will show that the lattice $\Lambda(\overline{McL})$ is a sublattice of the shorter Leech lattice $\Lambda_{23}$.

First we define some vectors in $\mathbb{R}^{24}$. For convenience, the coordinates of the vectors in $\mathbb{R}^{24}$ are indexed by the elements of the set $\{\infty,0,1,\ldots,22\}$.
Moreover, the Steiner systems $S(5,8,24)$ and $S(4,7,23)$ have $\{\infty,0,1,\ldots,22\}$ and $\{0,1,\ldots,22\}$ as their point sets, respectively. By following \cite{Deza}, we find that the Leech lattice $\Lambda_{24}$ is generated by the following $760$ ($=1+759$) vectors:
\begin{itemize}
\item $\frac{1}{\sqrt{8}}(-3,1^{23})$,
\item $\frac{1}{\sqrt{8}}(2^8,0^{16})$, in which the positions of the eight $2$'s form a block of $S(5,8,24)$.
\end{itemize}
Note that the vector $\mathbf{a}_0:=\frac{1}{\sqrt{8}}(4,4,0,\ldots,0)$ is a minimal vector in $\Lambda_{24}$. We have, from \cite[p. 179]{JHC},  that
$$\Lambda_{23}=\left\{\mathbf{v}-\frac{(\mathbf{v},\mathbf{a}_0)}{(\mathbf{a}_0,\mathbf{a}_0)}\mathbf{a}_0\mid \mathbf{v}\in \Lambda_{24},(\mathbf{v},\mathbf{a}_0)\text{ is even.}\right\}$$
Let $\Delta$ be the set of the following $275$ vectors:
\begin{enumerate}
\item $\mathbf{a}_i:=\frac{1}{\sqrt{8}}(4,0,\ldots,0,4,0,\ldots,0)^T$, where the first $4$ is in the first position $(\infty)$ and second $4$ is in the $i$-th position, for $1\leq i\leq 22$.
\item $\mathbf{b}_{B}:=\frac{1}{\sqrt{8}}(2,2^7,0^{16})^T$, where the first $2$ is in the position $(\infty)$ and the positions of the seven other $2$'s form a block $B$ of $S(4,7,23)$ which contains $0$.
\item $\mathbf{c}_{B}:=\frac{1}{\sqrt{8}}(3,-1^7,1^{16})^T$, where the $3$ is in the position $(\infty)$ and the positions of the seven other $-1$'s form a block $B$ of $S(4,7,23)$ which does not contain $0$.
\end{enumerate}
Note that $\Delta$ is a subset of $\Lambda_{24}$ (see \cite{Deza}), and for any vector $\mathbf{u}$ in $\Delta$, $(\mathbf{u},\mathbf{a}_0)=2$. Hence, the lattice generated by the following set
$$\widetilde{\Delta}:=\{\mathbf{u}-\frac{1}{2}\mathbf{a}_0\mid \mathbf{u}\in\Delta\}$$ is a sublattice of $\Lambda_{23}$. Moreover, by Proposition \ref{s23} and Proposition \ref{constructcommcl}, we find that the matrix $N$ with the $275$ vectors in $\widetilde{\Delta}$ columns satisfies $N^TN=A(\overline{McL})+3I$. This shows that the lattice $\Lambda(\overline{McL})$ is a sublattice of $\Lambda_{23}$.

Since that the shorter Leech lattice $\Lambda_{23}$ is an unimodular lattice with dimension $23$, it is $4$ integrable (see \cite[Theorem 18]{CS2}). This completes the proof.
\end{proof}

\section*{Acknowledgments}
We would like to thank Akihiro Munemasa for pointing out that one can extend the complement of the McLaughlin graph to obtain a slightly better lower bound.

\section*{Address:}
\textbf{a)}   Wen-Tsun Wu Key Laboratory of CAS, School of Mathematical Sciences, University
of Science and Technology of China, Hefei, Anhui, 230026, P.R. China \\

\noindent \textbf{b)} School of Mathematical Sciences, University of Science and Technology of China,
Hefei, Anhui, 230026, P.R. China\\

\noindent \textbf{c)} School of Mathematical Sciences, University of Science and Technology of China,
Hefei, Anhui, 230026, P.R. China\\
\section*{Email Address:}
koolen@ustc.edu.cn (J. H. Koolen)
\newline
masoodqau27@gmail.com; masood@mail.ustc.edu.cn (M. U. Rehman)
\newline
xuanxue@mail.ustc.edu.cn (Q. Yang)
\end{document}